\documentclass[11pt]{amsart}
\usepackage[utf8]{inputenc}
\usepackage{amsmath}
\usepackage{amsfonts}
\usepackage{amssymb}
\usepackage{amsthm}
\usepackage{csquotes}
\usepackage[english]{babel}
\usepackage[shortlabels]{enumitem}
\usepackage{lmodern}
\usepackage{cite}

\usepackage{color}

\newtheorem{theorem}{Theorem}[section]

\newtheorem{lemma}[theorem]{Lemma}

\newtheorem{ex}[theorem]{Example}

\theoremstyle{definition}

\theoremstyle{remark} \theoremstyle{remark}

\newcommand{\di}{\mbox{d}}
\newcommand{\summ}{\sum_{\alpha=1}^{s}}

\newcommand{\etaa}{\eta_{\alpha}}

\newcommand{\xib}{\xi_{\beta}}
\newcommand{\xia}{\xi_{\alpha}}

\newcommand{\xit}{\overline{\xi}}
\newcommand{\etat}{\overline{\eta}}
\def\a{\alpha}

\def\e{\eta}

\def\s{\sum}
\def\n{\nabla}

\numberwithin{equation}{section}

\begin{document}

\title[$f$-$(\kappa,\mu)$ manifolds]{Metric $f$-contact manifolds satisfying the $(\kappa,\mu)$-nullity condition}

\author[A. Carriazo]{Alfonso Carriazo}
 \address{Departamento de Geometr\'{i}a y Topolog\'{i}a, c/ Tarfia s/n, Universidad de Sevilla, Sevilla 41012, Spain}
 \email{carriazo@us.es}

\author[L. M. Fern\'andez]{Luis M. Fern\'andez}
 \address{Departamento de Geometr\'{i}a y Topolog\'{i}a, c/ Tarfia s/n, Universidad de Sevilla, Sevilla 41012, Spain}
 \email{lmfer@us.es}

\author[E. Loiudice]{Eugenia Loiudice}
  \address{Philipps Universit\"at Mar\-burg, Fach\-bere\-ich Math\-e\-matik und In\-for\-matik, Hans-Meerwein-Straße, 35032 Marburg, Germany}
 \email{loiudice@mathematik.uni-marburg.de}


\begin{abstract}
We prove that if the $f$-sectional curvature at any point $p$ of a $(2n+s)$-dimensional $f$-$(\kappa,\mu)$ manifold with $n>1$ is independent of the $f$-section at $p$, then it is constant on the manifold. Moreover, we also prove that an $f$-$(\kappa,\mu)$ manifold which is not an $S$-manifold is of constant $f$-sectional curvature if and only if $\mu=\kappa+1$ and we give an explicit expression for the curvature tensor field. Finally, we present some examples.
\end{abstract}

\maketitle

\section{Introduction.}

Riemannian manifolds with a complementary structure adapted to the metric have been widely studied, for instance, almost complex and almost contact manifolds. Both almost complex and almost contact structures are particular cases of $f$-structures introduced by K. Yano \cite{Y}.

A $(2n+s)$-dimensional Riemannian manifold $(M,g)$ endowed with an $f$-structure $f$ (that is, a tensor field of type (1,1) and rank $2n$ satisfying $f^3+f=0$ \cite{Y}) is said to be a {\it metric $f$-manifold} if, moreover, there exist $s$ global vector fields $\xi_1,\dots ,\xi_s$ on $M$ (called {\it structure vector fields}) and $s$ 1-foms $\eta_1,\dots ,\eta_s$ such that,
\begin{equation}\label{eta xi}
\etaa(\xib)=\delta_{\alpha}^{\beta}, \;f\xi_\alpha=0;\mbox{ }\e_\a\circ f=0;
\end{equation}
\begin{equation}\label{f^2}f^2=-I+\s_{\a=1}^s\e_\a\otimes\xi_\a;
\end{equation}
\begin{equation}\label{g(f,f)}
g(fX,fY)=g(X,Y)-\sum \etaa(X)\etaa(Y),
\end{equation}
for any $X,Y\in \mathcal{X}(M)$ and $\alpha,\beta=1,\dots,s$. Let $F$ be the 2-form on $M$ defined by $F(X,Y)=g(X,fY)$, for any $X,Y\in\mathcal{X}(M)$. Then, a metric $f$-manifold is said to be a {\it metric $f$-contact manifold} or an {\it almost $S$-manifold} (see \cite{CFF,DIP}) if $F=\di\e_\alpha$, for any $\alpha=1,\dots,s$ and a normal metric $f$-contact manifold is said to be an {\it S-manifold} (see \cite{B} for more details and examples).

The study of metric $f$-manifolds was initiated by D.E. Blair, S.I. Goldberg, K. Yano and J. Vanzura \cite{B,GY,V}. Later, J.L. Cabrerizo, L.M. Fern\'andez and M. Fern\'andez  \cite{CFF} studied metric $f$-contact manifolds as well as K. Duggal, S. Ianus and A.M. Pastore \cite{DIP} who called them almost $S$-manifolds. For manifolds with an $f$-structure, D.E. Blair \cite{B} introduced the notion of $K$-manifolds and their particular cases of $S$-manifolds and $C$-manifolds and proved that the space of a principal toroidal bundle over a Kaehler manifold is an $S$-manifold.

Moreover, $S$-structures are a natural generalization of Sasakian structures. However, unlike Sasakian manifolds, no $S$-structure can be realized on a simply connected compact manifold \cite{DL} (see also \cite[Corollary 4.3]{GL}). In \cite{TK}, an example of an even dimensional principal toroidal bundle over a Kaehler manifold which does not carry any Sasakian structure is presented and an $S$-structure on the even dimensional manifold $U(2)$ is constructed. Consequently and since it is well known that $U(2)$ does not admit a Kaehler structure, there exist manifolds such that the best structure which one can hope to obtain on them is an $S$-structure. In this context, it seems to be necessary to generalize to the setting of metric $f$-manifolds the concepts and results concerning almost contact geometry.

Following this idea, B. Capelletti Montano and L. Di Terlizzi generalize the concept of contact metric manifolds such that the characteristic vector field belongs to the $(\kappa,\mu)$-distribution, being $\kappa$ and $\mu$ real constants, studied in \cite{BKP} and classified in \cite{BO}, to metric $f$-contact manifolds (see \cite{CT}) by defining $f$-$(\kappa,\mu)$ manifolds. The purpose of this paper is to find conditions which characterizes $f$-$(\kappa,\mu)$ manifolds with constant $f$-sectional curvature.  We shall prove that if the $f$-sectional curvature at a point $p$ of a $(2n+s)$-dimensional $f$-$(\kappa,\mu)$ manifold with $n>1$ is independent of the $f$-section al $p$, then it is constant on the manifold. This result is analogous to Schur's theorem and extends a corresponding result on $S$-manifolds. Moreover, we shall also prove that an $f$-$(\kappa,\mu)$ manifold which is not an $S$-manifold is of constant $f$-sectional curvature if and only if $\mu=\kappa+1$ and an explicit expression for the curvature tensor field will be given.

These results generalize the corresponding ones in the case $s=1$, that is, in contact geometry \cite{K}.

Finally, we shall present some examples as application of the above results.

\section{Basic Definitions and Results.}

Let $(M,f,\xi_1,\dots,\xi_s,\eta_1,\dots,\eta_s,g)$ a metric $f$-manifold. The distribution on $M$ spanned by the structure vector fields is denoted by $\mathcal{M}$ and its complementary orthogonal distribution is denoted by $\mathcal{L}$. Consequently, $TM=\mathcal{L}\oplus\mathcal{M}$. Moreover, if $X\in\mathcal{L}$, then $\e_\a(X)=0$, for any $\a=1,\dots ,s$ and if $X\in\mathcal{M}$, then $fX=0$. From now on, we set $\overline{\xi}=\xi_1+\cdots+\xi_s$ and $\overline{\eta}=\eta_1+\cdots+\eta_s$.

Since $f$ is of rank $2n$, then $\eta_1\wedge\cdots\wedge\eta_s\wedge F^n\neq 0$ and, in particular, $M$ is orientable. A metric $f$-contact manifold is said to be a {\it metric $f$-$K$-contact manifold} if the structure vector fields are Killing vector fields.

The $f$-structure $f$ is said to be {\it normal} if
$$[f,f]+2\sum_{\alpha=1}^s\xi\alpha\otimes\di\eta_\alpha=0,$$
where $[f,f]$ denotes the Nijenhuis tensor of $f$.

On a metric $f$-contact manifold there are defined the $(1,1)$-tensor fields
$$h_\alpha=\frac{1}{2}L_{\xi_\alpha}f,\,\alpha=1,\dots,s,$$
(see \cite{CFF}), where $L_{\xi_\alpha}f$ is the Lie derivative of $f$ in the direction $\xi_\alpha$. These operators are self adjoint, traceless, anticommute with $f$ and
\begin{equation}\label{h}
h_\alpha\xi_\beta=0, \, \eta_\alpha\circ h_\beta=0,\,\alpha,\beta=1,\dots,s.
\end{equation}

Moreover, the structure vector field $\xi_a$ is a Killing vector field if and only if $h_\alpha=0$ \cite[Theorem 2.6]{CFF}.

Given $\kappa,\mu$ two real constants, a metric $f$-contact manifold is said to verify the {\it $(\kappa,\mu)$-nullity condition} (see \cite{CT}) or to be a {\it metric  $f$-$(\kappa,\mu)$ manifold}, if for each $\alpha\in\{1,\dots, s\}$ and $X,Y\in\mathcal{X}(M)$,
\begin{equation}\label{16}
R(X,Y)\xi_\alpha=\kappa\Big(\overline{\eta}(X)f^2Y-\overline{\eta}(Y)f^2X\Big)+\mu\Big(\overline{\eta}(Y)h_\alpha X-\overline{\eta}(Y)h_\alpha Y\Big),
\end{equation}
where $R$ is denoting the curvature tensor field of $g$. B. Cappelletti Montano and L. Di Terlizzi \cite{CT} proved that, in this context, $\kappa\leq 1$ and, if $\kappa<1$, any $h_\alpha$ has eigenvalues $0,\pm\sqrt{1-\kappa}$ and $h_1=\cdots=h_s$. Putting $h$ as this common value and by using the symmetry properties of $R$ and the symmetry of $f^2$, (\ref{16}) becomes
\begin{equation}\label{1.12}
 R(\xia,X)Y=\kappa\Big(\etat(Y)f^2 X-g(X,f^2 Y)\xit\Big)+\mu\Big(g(X,hY)\xit-\etat(Y)hX\Big),
\end{equation}
for any $\alpha=1,\dots,s$.

Moreover, in this case of being $\kappa<1$, denoting by $\mathcal{L}_+$ and $\mathcal{L}_-$ the $n$-dimensonal distributions of the eigenspaces of $\lambda=\sqrt{1-\kappa}$ and $-\lambda$, respectively, it is known (cf. \cite{CT}) that $\mathcal{L}_+$ and $\mathcal{L}_-$ are mutually ortogonal and, since $f$ anticommutes with $h$, $f(\mathcal{L}_+)=\mathcal{L}_-$ and $f(\mathcal{L}_-)=\mathcal{L}_+$.

In the same paper \cite{CT}, it is also proved that a metric $f$-contact manifold is an $S$-manifold if and only if $\kappa=1$.

For later use, we recall that, from Lemma 2.7 of \cite{CT}, in a metric $f$-contact manifold satisfying the $(\kappa,\mu)$-nullity condition,
\begin{equation}\label{Rf}
\begin{aligned}
    R(X,Y)fZ ={}&f R(X,Y)Z +\Big(\kappa\big(\overline{\eta}(Y)g(fX,Z)-\overline{\eta}(X)g(fY,Z)\big)\\
              &+\mu \big(\overline{\eta}(Y)g(f h X,Z)-\overline{\eta}(X)g(f h Y,Z)\big)\Big)\overline{\xi} \\
              & + s \Big( -g(h Y-f^2Y,Z)(fX+ fhX)+g(h X-f^2X,Z)(fY+fh Y)\\
                        &-g(fY+fh Y,Z)(h X-f^2X)+g(fX+fh X,Z)(h Y-f^2Y)\Big)\\
                        & +\overline{\eta}(Z) \Big( \kappa\big(\overline{\eta}(X)fY-\overline{\eta}(Y)fX\big) +\mu\big(\overline{\eta}(X)fh Y-\overline{\eta}(Y)fh X\big)\Big),
\end{aligned}
\end{equation}
for any $X,Y,Z\in\mathcal{X}(M)$. Moreover, in the same conditions, if $Q$ denotes the Ricci operator and $\kappa<1$, from Corollary 2.1 of \cite{CT}:
\begin{equation}\label{Ricci}
 Q=s\Big(2(1-n)+n\mu\Big)f^2+s\Big(2(n-1)+\mu\Big)h+2n\kappa\overline{\eta}\otimes\overline{\xi}.
\end{equation}

\section{Main Theorems.}

With a technique similar to the one used by T. Koufogiorgos in \cite[Theorem 4.1]{K} (in case of metric contact $(\kappa,\mu)$-spaces, that is, in case of $s=1$), we prove the following theorem, which generalizes both \cite[Theorem 4.1]{K} and \cite[Proposition 1.5]{Kobayashi-Tsuchiya} (see also \cite[Theorem 2.1]{Fernandez}).
\begin{theorem}
Let $(M,f,\xi_1,\dots,\xi_s,\eta_1,\dots,\eta_s,g)$ be a $(2n+s)$-dimensional metric $f$-$(\kappa,\mu)$ manifold, with $n>1$. If for every point $p\in M$ the $f$-sectional curvature at $p$ is constant, then it is constant on $M$ (so, $M$ is called a {\it $f$-$(\kappa,\mu)$-space-form}) and the curvature tensor of $M$ is given by
\begin{equation}\label{f}
 \begin{aligned}
  4R(X,Y)Z ={} &(H+3s)\Big(g(f^2Y,Z)f^2X-g(f^2X,Z)f^2Y \Big)+(H-s)\Big(2g(f Y,X)f Z\\
             &+g(X,f Z)f Y -g(Y,f Z)f X \Big)-2s \Big(g(hX, Z)h Y -g(hY, Z)h X  \\
             &-g(f hX, Z)f h Y +g(f hY, Z)f h X -2g(f^2X,Z)hY+2g(f^2Y,Z)hX \\
             &-2g(hX, Z)f^2 Y + 2g(hY, Z)f^2 X\Big) +4 \kappa \Big( \etat (X) \etat (Z)f ^2 Y\\
             &-  \etat (X)g(Y,f ^2 Z)\xit - \etat (Y) \etat (Z)f ^2 X+ \etat (Y) g(X,f ^2 Z)\xit \Big)\\
             & +4 \mu \Big( \etat (Y) \etat (Z)hX-  \etat (Y)g(X,h Z)\xit - \etat (X) \etat (Z)hY+ \etat (X) g(Y,h Z)\xit  \Big),
 \end{aligned}
 \end{equation}
where $H$ is the constant $f$-sectional curvature of $M$. In particular if $\kappa <1$, then $\mu=\kappa+1$ and $H=-s(2\kappa+1)$.
\end{theorem}
\begin{proof}
For $\kappa =1$, that is, when $M$ is an $S$-manifold, the theorem is known (see \cite[Proposition 1.5]{Kobayashi-Tsuchiya} and \cite[Theorem 2.1]{Fernandez}). Thus, suppose $\kappa <1$.

Let $p\in M$ and $X,Y\in \mathcal{L}_p$. Then, using basic curvature identities, equations \eqref{h}, \eqref{Rf} and the fact that $M$ a is metric $f$-contact manifold, we obtain:
 \begin{equation}\label{g1}
 \begin{aligned}
   g(R(X,fX)Y,fY)={}& -g(R(fX,Y)X,fY)-g(R(Y,X)fX,fY)\\
                 ={}& g(R(X,fY)Y,fX)+g(fR(X,Y)X,fY)\\
                  &+s\Big(g(fh X,X)g(h Y,fY)-g(fY+fh Y,X)g(h X+X,fY)\\
                  &-g(h Y+Y,X)g(fX+ fhX,fY)\\
                  &+g(h X+X,X)g(fY+fh Y,fY)\Big)\\
                 ={}& g(R(X,fY)Y,fX)+g(R(X,Y)X,Y)+s\Big(-g(fh X,X)g(fh Y,Y)\\
                 &-g(fX,Y)^2+g(fhX,Y)^2 -2g(X,Y)g(hX,Y)-g(hX,Y)^2\\
                 &-g(X,Y)^2+g(X,hX)g(Y,Y)+g(X,hX)g(Y,hY)\\
                 &+g(X,X)g(Y,Y)+g(X,X)g(hY,Y)\Big).
 \end{aligned}
 \end{equation}

Analogously we have
\begin{equation}\label{g2}
 \begin{aligned}
   g(R(X,fY)X,fY)={}& g(R(X,fY)Y,fX)+s\Big(g(X,Y)^2-g(hX,Y)^2\\
                  &-g(fh X,X)g(fh Y,Y)-g(X,X)g(Y,Y)-g(X,hX)g(Y,Y)\\
                  &+g(X,X)g(hY,Y)+g(X,hX)g(Y,hY)+g(fX,Y)^2\\
                  &+g(fhX,Y)^2+2g(fX,Y)g(fhX,Y)\Big),
 \end{aligned}
 \end{equation}
\begin{equation}\label{g3}
 \begin{aligned}
   g(R(Y,fX)Y,fX)={}& g(R(X,fY)Y,fX)+s\Big(g(X,Y)^2-g(hX,Y)^2-g(hX,Y)^2\\
                  &-g(fh X,X)g(fh Y,Y)+g(fX,Y)^2+g(fhX,Y)^2\\
                  &-g(X,X)g(Y,Y)-2g(fX,Y)g(fhX,Y)-g(X,X)g(hY,Y)\\
                  &+g(X,hX)g(Y,Y)+g(X,hX)g(Y,hY)\Big)
 \end{aligned}
 \end{equation}
and:
\begin{equation}\label{g4}
 \begin{aligned}
   g(R(X,Y)fX,fY)={}& g(R(X,Y)X,Y)+s\Big(-g(X,Y)^2-g(hX,Y)^2-2g(X,Y)g(hX,Y)\\
                 &+g(X,X)g(Y,Y)+g(X,X)g(hY,Y)+g(X,hX)g(Y,Y)\\
                 &+g(X,hX)g(Y,hY)+g(fX,Y)^2+g(fhX,Y)^2\\
                 &-g(fh X,X)g(fh Y,Y)\Big).
 \end{aligned}
 \end{equation}

If $H(p)$ denotes the value of the $f$-sectional curvature at $p$, then for any $X,Y \in \mathcal{L}_p$:
\begin{equation*}
\begin{aligned}
 &g(R(X+Y,fX+fY)(X+Y),fX+fY)=-H(p)g(X+Y,X+Y)^2,\\
 &g(R(X-Y,fX-fY)(X-Y),fX-fY)=-H(p)g(X-Y,X-Y)^2.
\end{aligned}
 \end{equation*}

Summing these equations and using \eqref{g1}, \eqref{g2}, \eqref{g3}, we get:
\begin{equation}\label{H1}
\begin{aligned}
  -H(p)\Big(2g(X,Y)^2+g(X,X)g(Y,Y)\Big)={}& 2g(R(X,fX)Y,fY)+g(R(X,fY)X,fY)\\
                                        &+2g(R(X,fY)Y,fX)+g(R(Y,fX)Y,fX)\\
                                       ={}&3g(R(X,fY)Y,fX)+g(R(X,Y)X,Y)\\
                                        &+s\Big(-2g(hX,Y)^2-2g(X,Y)g(hX,Y)\\
                                        &+g(X,X)g(hY,Y)+g(Y,Y)g(hX,X)\\
                                        &+2g(hX,X)g(hY,Y)+2g(fhX,Y)^2\\
                                        &-2g(fhX,X)g(fhY,Y)\Big).
 \end{aligned}
 \end{equation}

Now we replace $Y$ by $fY$ in the previous equation, obtaining
 \begin{equation}\label{H2}
\begin{aligned}
  -H(p)\Big(2g(X,fY)^2+g(X,X)g(Y,Y)\Big)={}&-3g(R(X,Y)fY,fX)+g(R(X,fY)X,fY)\\
                                        &+s\Big(-2g(fhX,Y)^2+2g(X,fY)g(fhX,Y)\\
                                        &-g(X,X)g(hY,Y)+g(Y,Y)g(hX,X)\\
                                        &-2g(hX,X)g(hY,Y)+2g(hX,Y)^2\\
                                        &+2g(fhX,X)g(fhY,Y)\Big)\\
                                        ={}&g(R(X,fY)Y,fX)+3g(R(X,Y)X,Y)+\\
                                        &+s\Big(-2g(X,Y)^2-2g(hX,Y)^2\\
                                        &-2g(fhX,X)g(fhY,Y)+2g(X,X)g(Y,Y)\\
                                        &+3g(hX,X)g(Y,Y)+3g(X,X)g(hY,Y)\\
                                        &+2g(hX,X)g(hY,Y)-2g(fX,Y)^2\\
                                        &+2g(fhX,Y)^2-6g(X,Y)g(hX,Y)\Big),
 \end{aligned}
 \end{equation}
where we have used \eqref{g(f,f)} and \eqref{h} for the first equality and \eqref{g2} and \eqref{g4} for the second. Combining \eqref{H1} and \eqref{H2}, we deduce:
  \begin{equation*}
 \begin{aligned}
  4 g(R(X,Y)Y,X)={}& (H(p)+3s)\Big(g(X,X)g(Y,Y)-g(X,Y)^2\Big)+3(H(p)-s)g(X,fY)^2\\
                 &-2s\Big(g(hX,Y)^2+4g(X,Y)g(hX,Y)-2g(X,X)g(hY,Y)\\
                 &-2g(Y,Y)g(hX,X)-g(hX,X)g(hY,Y)-g(fhX,Y)^2\\
                 &+g(fhX,Y)g(fhY,Y)\Big).
 \end{aligned}
 \end{equation*}

Now, using this identity in
 \begin{equation*}
  g(R(X+Z,Y)Y,X+Z)=g(R(X,Y)Y,X)+g(R(Z,Y)Y,Z)+2g(R(X,Y)Y,Z),
 \end{equation*}
where $X,Y,Z\in \mathcal{L}_p$, we obtain, by a straightforward calculation:
\begin{equation*}
 \begin{aligned}
  4 g(R(X,Y)Y,Z)={}& (H(p)+3s)\Big(g(X,Z)g(Y,Y)-g(X,Y)g(Z,Y)\Big)\\
                 &+3(H(p)-s)g(X,fY)g(Z,fY)-2s\Big(g(hX,Y)g(hZ,Y)\\
                 &+2g(X,Y)g(hZ,Y)+2g(Z,Y)g(hX,Y)-2g(X,Z)g(hY,Y)\\
                 &-2g(Y,Y)g(hZ,X)-g(hX,Z)g(hY,Y)-g(fhX,Y)g(fhZ,Y)\\
                 &+g(fhX,Z)g(fhY,Y)\Big).
 \end{aligned}
 \end{equation*}

It is easy to check, by using the (\ref{16}), that the previous equation also holds for any $Z\in T_pM$ and $X,Y\in \mathcal{L}_p$. Thus, for each $X,Y\in \mathcal{L}_p$:
  \begin{equation*}
 \begin{aligned}
  4R(X,Y)Y ={} &(H(p)+3s)\Big(g(Y,Y)X-g(X,Y)Y \Big)+3(H(p)-s)g(f Y,X)f Y\\
             &-2s \Big(g(hX, Y)h Y +2g(X,Y)h Y  +2g(hX, Y)Y -2g( hY, Y) X \\
             &-2g(Y,Y)hX-g(hY,Y)hX -g(fhX, Y)fh Y + g(fhY, Y)fh X\Big).
 \end{aligned}
 \end{equation*}

Next, let $X,Y,Z\in \mathcal{L}_p$. Then, using the above identity in
  \begin{equation*}
   R(X,Y+Z)(Y+Z)=R(X,Y)Y+R(X,Y)Z+R(X,Z)Z+R(X,Z)Y,
  \end{equation*}
we obtain:
\begin{equation}\label{4R}
 \begin{aligned}
  4\Big(R(X,Y)Z+R(X,Z)Y\Big) ={} &(H(p)+3s)\Big(2g(Y,Z)X-g(X,Y)Z-g(X,Z)Y \Big)\\
             &+3(H(p)-s)\Big(g(f Y,X)f Z+g(f Z,X)f Y\Big)\\
             &-2s \Big(g(hX, Z)h Y +2g(X,Z)h Y  +g(hX, Y)hZ \\
             &+2g( X, Y) hZ+2g(hX,Z)Y+2g(hX,Y)Z\\
             &-4g(hY,Z)X-4g(Y,Z)hX-2g(hY,Z)hX\\
             & -g(fhX, Y)fh Z -g(fhX, Z)fh Y+2 g(fhY, Z)fh X\Big).
 \end{aligned}
 \end{equation}

Replacing $X$ by $Y$ and $Y$ by $-X$ in \eqref{4R} we have:
 \begin{equation}\label{4RR}
 \begin{aligned}
  4\Big(R(X,Y)Z+R(Z,Y)X\Big) ={} &(H(p)+3s)\Big(-2g(X,Z)Y+g(X,Y)Z+g(Y,Z)X \Big)\\
             &+3(H(p)-s)\Big(-g(f X,Y)f Z-g(f Z,Y)f X\Big)\\
             &-2s \Big( -g(hY, Z)hX-2g( Y,Z) hX-g(hY, X)h Z  \\
             &-2g(X,Y)h Z-2g(hY,Z)X-2g(X,hY)Z\\
             &+4g(hX,Z)Y+4g(X,Z)hY+2g(hX,Z)hY\\
             &+g(fhY, X)fh Z +g(fhY, Z)fh X-2 g(fhX, Z)fh Y\Big).
 \end{aligned}
 \end{equation}

Summing \eqref{4R} and \eqref{4RR} and by using the Bianchi's first identity and the fact that $fh$ is a symmetric operator and $f$ is antisymmetric, we obtain:
 \begin{equation}\label{rr}
 \begin{aligned}
  4R(X,Y)Z ={} &(H(p)+3s)\Big(g(Y,Z)X-g(X,Z)Y \Big)+(H(p)-s)\Big(2g(f Y,X)f Z\\
               &+g(f Z,X)f Y+g(f Y,Z)f X\Big)-2s \Big(g(hX, Z)h Y +2g(X,Z)h Y  \\
             & +2g(hX, Z)Y -2g( hY, Z) X-2g(Y,Z)hX-g(hY,Z)hX \\
             &-g(fhX, Z)fh Y + g(fhY, Z)fh X\Big).
 \end{aligned}
 \end{equation}
At this point, again one can easily check, using (\ref{16}) and \eqref{h}, that the above equation is also valid for any $Z\in T_pM$ and $X,Y\in \mathcal{L}_p$. Now, we consider any $X,Y,Z\in T_pM$. We have that
 \begin{equation*}
  X=X^H+\summ \etaa(X)\xia, \; Y=Y^H+\summ \etaa(Y)\xia,
 \end{equation*}
where $X^H,Y^H\in \mathcal{L}_p$ and that:
 \begin{equation*}
  R(X,Y)Z=R(X^H,Y^H)Z+ \summ \etaa(X)R(\xia,Y)Z+\summ \etaa(Y)R(X^H,\xia)Z.
 \end{equation*}

Consequently, using \eqref{rr} and \eqref{1.12} in the above equation, we finally obtain \eqref{f}.

Next, we prove that the $f$-sectional curvature of $M$ is constant. Let $\{e_i\}$ be a local orthonormal basis of tangent vector fields on $U\subset M$. Then, taking $Y=Z=e_i$ in \eqref{f} and summing over $i$ we obtain the following identity for the Ricci operator at $p\in U$,
 \begin{equation}\label{Q}
   2Q=-\Big((n+1)H(p)+3s(n-1)+2s\kappa\Big) f^2 + 4\kappa n \,\overline{\eta}\otimes\overline{\xi}-2s\Big(2(1-n)-\mu\Big)h,
 \end{equation}
where we have used \eqref{f^2}, \eqref{eta xi}, \eqref{h}, the antisymmetry of $f$ and the fact that, since $h$ is traceless, then,
$$\text{Tr}h=\sum_{i=1}^{2n+s}g(h e_i,e_i)=0$$
and:
$$\text{Tr}fh=\sum_{i=1}^{2n+s}g(fh e_i,e_i)=0.$$

Comparing \eqref{Q} and \eqref{Ricci} we have that
 \begin{equation}\label{H}
  (n+1)H(p)=s(n-1-2\mu n-2\kappa)
 \end{equation}
and the $f$-sectional curvature is constant. By \cite[Theorem 2.3]{CT} we know that
\begin{equation*}
 K(X,fX)=-s(\kappa+\mu)\left(g(X,f^2X)\right)^2=-s(\kappa+\mu),
\end{equation*}
for any $X\in\mathcal{L}_+$ with $g(X,X)=1$. Thus, equation \eqref{H} becomes
\begin{equation*}
 -(n+1)(\kappa+\mu)=n-1-2\mu n-2\kappa,
\end{equation*}
namely $(\kappa-\mu+1)(n-1)=0$. Hence, since $n>1$, we have that $\mu=\kappa+1$ and $H=-s(2\kappa+1)$.
\end{proof}

We observe that the condition $n>1$ in the above theorem is necessary since, for $n=1$, there are examples of flat metric $f$-contact manifolds (see \cite[Example 6.2]{TKP}) and hence of $f$-$(\kappa,\mu)$-space-forms with $\mu=0\neq 1=\kappa+1$.

Moreover, in the above theorem, we have proved that $\mu=\kappa+1$ in the case of being $\kappa<1$, that is, if the $f$-$(\kappa,\mu)$-space-form is not an $S$-manifold. The converse is also true and we are going to prove that if $(M,f,\xi_1,\dots,\xi_s,\eta_1,\dots,\eta_s,g)$ ($n>1$) is a metric $f$-contact manifold satisfying the $(\kappa,\kappa+1)$-nullity condition and it is not an $S$-manifold, that is, $\kappa<1$, then it has constant $f$-sectional curvature.

To that end, firstly we state the following lemma which is proved by using \cite[Theorem 2.2]{CT}, \eqref{g(f,f)} and a long straightforward computation.
\begin{lemma} Let $M$ be an $f$-$(\kappa,\mu)$ manifold which is not an $S$-manifold. Let $X\in\mathcal{L}$ be a unit vector field and put $X=X_++X_-$, where $X_+\in\mathcal{L}_+$ and $X_-\in\mathcal{L}_-$. Then,
\begin{equation}\label{f-sect}
H(X)=-s(\kappa+\mu)+4s(\kappa-\mu+1)\Big(g(X_+,X_+)g(X_-,X_-)-g(X_+,fX_-)^2\Big),
\end{equation}
where $H(X)=K(X,fX)$ denotes the $f$-sectional curvature determined by a unit vector field $X\in\mathcal{L}$.
\end{lemma}

Thus, we have:
\begin{theorem}\label{T33}
Let $(M,f,\xi_1,\dots,\xi_s,\eta_1,\dots,\eta_s,g)$ be a $(2n+s)$-dimensional $f$-$(\kappa,\mu)$ manifold with $n>1$ which is not an $S$-manifold. Then, $M$ is an $f$-$(\kappa,\mu)$-space form if and only if $\mu=\kappa+1$.
\end{theorem}
\begin{proof}
We only need to prove that if $\mu=\kappa+1$, then $M$ has constant $f$-sectional curvature. But, from \eqref{f-sect} we obtain $H(X)=-s(2k+1)$, for any unit vector field $X\in\mathcal{L}$.
\end{proof}
\section{Examples.}
\begin{ex}{\it Generalized $S$-space-forms.} {\rm A metric $f$-manifold with two structure vector fields
$$(M,f,\xi_1,\xi_2,\eta_1,\eta_2,g),$$
is said to be a {\it generalized $S$-space form} if there exist seven differentiable functions on $M$, $F_1,\dots,F_7$ such that the curvature tensor field $R$ of $M$ satisfies
\begin{equation}\label{gssf}
\begin{split}
R(X,Y)Z=&F_1\left\{g(Y,Z)X-g(X,Z)Y\right\}\\
&+F_2\left\{g(X,fZ)fY-g(Y,fZ)fX+2g(X,fY)fZ\right\}\\
&+F_3\left\{\eta_1(X)\eta_1(Z)Y-\eta_1(Y)\eta_1(Z)X+g(X,Z)\eta_1(Y)\xi_1-g(Y,Z)\eta_1(X)\xi_1\right\}\\
&+F_4\left\{\eta_2(X)\eta_2(Z)Y-\eta_2(Y)\eta_2(Z)X+g(X,Z)\eta_2(Y)\xi_2-g(Y,Z)\eta_2(X)\xi_2\right\}\\
&+F_5\left\{\eta_1(X)\eta_2(Z)Y-\eta_1(Y)\eta_2(Z)X+g(X,Z)\eta_1(Y)\xi_2-g(Y,Z)\eta_1(X)\xi_2\right\}\\
&+F_6\left\{\eta_2(X)\eta_1(Z)Y-\eta_2(Y)\eta_1(Z)X+g(X,Z)\eta_2(Y)\xi_1-g(Y,Z)\eta_2(X)\xi_1\right\}\\
&+F_7\left\{\eta_1(X)\eta_2(Y)\eta_2(Z)\xi_1-\eta_2(X)\eta_1(Y)\eta_2(Z)\xi_1\right.\\
&+\left.\eta_2(X)\eta_1(Y)\eta_1(Z)\xi_2-\eta_1(X)\eta_2(Y)\eta_1(Z)\xi_2\right\},
\end{split}
\end{equation}
for any $X,Y,Z\in\mathcal{X}(M)$ \cite{CFF2,T}. Then, if $M$ is also a metric $f$-contact manifold, a direct expansion from \eqref{gssf} shows that $M$ is a $f$-$(\kappa,\mu)$ manifold if the functions $F_1-F_3,F_5,F_6,F_4-F_7$ are constant functions and $F_1-F_3=-F_5=-F_6=F_4-F_7$. In such a case, $\kappa=F_1-F_3$ and $\mu=0$. Some examples of generalized $S$-space-forms satisfying these conditions can be found in \cite{CFF2}.

Consequently, from Theorem \ref{T33} and \cite[Theorem 4.3]{CFF2}, we obtain:
\begin{theorem} A generalized $S$-space form with two structure vector fields verifying the $(\kappa,\mu)$-nullity condition is a $f$-$(\kappa,\mu)$-space-form if and only if it is either an $S$-space-form or $\kappa=-1, \mu=0$.
\end{theorem}}
\end{ex}
\begin{ex} {\it Trans-$S$-manifolds.} {\rm A $(2n+s)$-dimensional metric $f$-manifold $M$ is said to be an {\it almost trans-$S$-manifold} if it satisfies \begin{equation}\label{nabla}
\begin{split}
(\nabla_Xf)Y=\sum_{i=1}^s&\Big[\alpha_i\{g(fX,fY)\xi_i+\eta_i(Y)f^2X\}\\
&+\beta_i\{g(fX,Y)\xi_i-\eta_i(Y)fX\}\Big],
\end{split}
\end{equation}
for certain smooth functions (called the {\it characteristic functions}) $\alpha_i,\beta_i$, $i=1,\dots,s$, on $M$ and any $X,Y\in\mathcal{X}(M)$. If, moreover, $M$ is normal, then it is said to be a {\it trans-$S$-manifold}.

On the other hand, for the curvature tensor field $R$ of a metric $f$-$K$-contact manifold it is known that \cite{CFF}
\begin{equation}\label{t421}
R(X,\xi_\a)Y=-(\n_Xf)Y,
\end{equation}
for any $X,Y\in\mathcal{X}(M)$ and $\a=1,\dots ,s$.

Consequently, from \eqref{nabla} and \eqref{t421}, we deduce that a metric $f$-$K$-contact manifold satisfies the $(\kappa,\mu)$-nullity condition if and only if $\alpha_1=\cdots=\alpha_s=\kappa$ (and so, all the characteristic functions $\alpha_i$, $i=1,\dots,s$ are constant functions) and $\beta_1=\cdots=\beta_s=0$, for any $\mu$.

Thus, from Theorem \ref{T33} and \cite[Corollary 3.2]{AFP} we can prove:
\begin{theorem} Let $M$ be a trans-$S$-manifold, whit characteristic functions $\alpha_i,\beta_i$, $i=1,\dots,s$,  which is also a metric $f$-$K$-contact manifold. Then, $M$ is an $f$-$(\kappa,\mu)$-space-form if and only if it is either an $S$-space-form or $\alpha_1=\cdots=\alpha_s=\kappa$, $\beta_1=\cdots=\beta_s=0$ and $\mu=\kappa+1$.
\end{theorem}}
\end{ex}
\begin{ex}{\rm Given a metric $f$-manifold $$(M,f,\xi_1,\dots,\xi_s,\eta_1,\dots,\eta_s,g),$$
if we consider a {\it $D$-homothetic deformation of constant $a>0$}
\begin{equation*}\label{dconf}
\widetilde f=f, \quad \widetilde\xi_\alpha=\frac{1}{a}\xi_\alpha, \quad \widetilde\eta_\alpha=a\eta_\alpha, \quad
\widetilde g= ag + a(a-1)\sum_{\alpha=1}^s\eta_\alpha\otimes\eta_\alpha,
\end{equation*}
for any $\alpha=1,\dots,s$, then, it is easy to prove that $(M,\widetilde f,\widetilde\xi_1,\dots,\widetilde\xi_s,\widetilde\eta_1,\dots,\widetilde\eta_s,\widetilde g)$ is also a metric $f$-manifold.

Performing a $D$-homothetic deformation of constant $a>0$ to a $(2n+s)$-dimensional ($n>1$) metric $f$-contact manifold $M$ satisfying $R(X,Y)\xi_\alpha=0$ for any $X,Y\in\mathcal{X}(M)$ and $\alpha=1,\dots,s$ (this manifold is locally isometric to $E^{n+s}\times S^n$, \cite[Theorem 2.1]{T}), we obtain on $M$ a metric $f$-$(\kappa,\mu)$-structure with $\kappa=(a^2-1)/a^2$ and $\mu=2(a-1)/a$. Thus, the condition $\mu=\kappa+1$ is equivalent to $a=1/2$. Then, from Theorem \ref{T33}, for $a=1/2$, $M$ is an $f$-$(\kappa,\mu)$-space-form.

In case $n=1$, from \cite[Theorem 2.2]{DT}, the condition $R(X,Y)\xi_\alpha=0$ implies that $M^{2+s}$ is flat and so, these manifolds are $f$-$(\kappa,\mu)$-space-forms with $\kappa=\mu=0$. Performing a $D$-homothetic deformation of such manifolds with constant $a\neq 1/2$, $a\neq 1$, we obtain examples of $f$-$(\kappa,\mu)$-space-forms with non-zero $f$-sectional curvature and $\mu\neq\kappa+1$. In fact, for the deformed manifold, $\kappa=(a^2-1)/a^2$, $\mu=2(a-1)/a$ and, from \cite[Theorem 2.3]{CT}, $H=-s(\kappa+\mu)=-s(3a^2-2a-1)/a^2$.
}\end{ex}

\end{document}